\documentclass[11pt, a4paper]{amsart}
\usepackage{a4wide}
\usepackage{graphicx,enumitem}
\usepackage{subfigure}
\usepackage{units}
\usepackage{color}

\usepackage{ulem}
\usepackage{latexsym}

\usepackage[pdftex,
            pdftitle={SPDE weak error simulation by MC vs.\ MLMC},
            pdfauthor={A. Lang and A. Petersson},
            bookmarksopen,
            colorlinks,
            linkcolor=black,
            urlcolor=black,
	    citecolor=black
]{hyperref}

\usepackage[norefs, nocites]{refcheck} % For catching unused labels etc.
\usepackage{dsfont}

\usepackage{mathtools}
\DeclarePairedDelimiter{\ceil}{\lceil}{\rceil}

\newcommand{\R}{\mathbb{R}}

\newcommand{\N}{\mathbb{N}}

\newcommand{\gs}{\sigma}

\newcommand{\gz}{\zeta}

\newcommand{\gO}{\Omega}

\newcommand{\cA}{\mathcal{A}}

\newcommand{\cF}{\mathcal{F}}

\newcommand{\Op}{\operatorname{O}}

\DeclareMathOperator{\E}{\mathbb{E}} %expectation
\DeclareMathOperator{\Var}{\mathsf{Var}} %variance

\newcommand{\dd}{\mathrm{d}}

\newcommand{\bigE}[1]{\E\left[#1\right]}

\newcommand{\smallE}[1]{\E[#1]}
\newcommand{\inpro}[3]{ \left\langle #1 , #2 \right\rangle_{#3} }
\newcommand{\hsspacetwo}[2]{ L_{\textsc{HS}}( #1 ; #2 ) }
\newcommand{\norm}[2]{\| #1 \|_{#2}}
\newcommand{\bignorm}[2]{\left|\left| #1 \right|\right|_{#2}}
\newcommand{\KL}{Karhunen--Lo\`eve }

\newtheorem{lemma}{Lemma}[section]
\newtheorem{proposition}[lemma]{Proposition}

\newtheorem{theorem}[lemma]{Theorem}
\newtheorem{corollary}[lemma]{Corollary}
\theoremstyle{remark}

\theoremstyle{definition}

\newtheorem{assumption}[lemma]{Assumption}
\begin{document}
\title[SPDE weak error simulation by MC vs.\ MLMC]{
Monte Carlo versus multilevel Monte Carlo in weak error simulations of SPDE approximations
}

\author[A.~Lang]{Annika Lang} \address[Annika Lang]{\newline Department of Mathematical Sciences
\newline Chalmers University of Technology \& University of Gothenburg
\newline S--412 96 G\"oteborg, Sweden.} \email[]{annika.lang@chalmers.se}

\author[A.~Petersson]{Andreas Petersson} \address[Andreas Petersson]{\newline Department of Mathematical Sciences
\newline Chalmers University of Technology \& University of Gothenburg
\newline S--412 96 G\"oteborg, Sweden.} 
\email[]{andreas.petersson@chalmers.se}

\thanks{
Acknowledgement. 
The authors wish to express many thanks to Stig Larsson, Christoph Schwab, and two anonymous referees 
for fruitful discussions and helpful comments.
The work was supported in part by the Swedish Research Council under Reg.~No.~621-2014-3995 as well as the Knut and Alice Wallenberg foundation. The simulations were performed on resources at Chalmers Centre for Computational Science and Engineering (C3SE) provided by the Swedish National Infrastructure for Computing (SNIC)
}
% \date{September 26, 2016}
 
\begin{abstract}
The simulation of the expectation of a stochastic quantity $\E[Y]$ by Monte Carlo methods is known to be computationally expensive especially if the stochastic quantity or its approximation~$Y_n$ is expensive to simulate, e.g., the solution of a stochastic partial differential equation. If the convergence of $Y_n$ to $Y$ in terms of the error $|\E[Y - Y_n]|$ is to be simulated, this will typically be done by a Monte Carlo method, i.e., $|\E[Y] - E_N[Y_n]|$ is computed. In this article upper and lower bounds for the additional error caused by this are determined and compared to those of $|E_N[Y - Y_n]|$, which are found to be smaller. Furthermore, the corresponding results for multilevel Monte Carlo estimators, for which the additional sampling error converges with the same rate as $|\E[Y - Y_n]|$, are presented. Simulations of a stochastic heat equation driven by multiplicative Wiener noise and a geometric Brownian motion are performed which confirm the theoretical results and show the consequences of the presented theory for weak error simulations.
%%%%%%%%%%%%%%%%%%%%%%%%%%%%%%%%%%%%%%%%%%
% text version of abstract
%%%%%%%%%%%%%%%%%%%%%%%%%%%%%%%%%%%%%%%%%%
% The simulation of the expectation of a stochastic quantity E[Y] by Monte Carlo methods is known to be computationally expensive especially if the stochastic quantity or its approximation Y_n is expensive to simulate, e.g., the solution of a stochastic partial differential equation. If the convergence of Y_n to Y in terms of the error |E[Y - Y_n]| is to be simulated, this will typically be done by a Monte Carlo method, i.e., |E[Y] - E_N[Y_n]| is computed. In this article upper and lower bounds for the additional error caused by this are determined and compared to those of |E_N[Y - Y_n]|, which are found to be smaller. Furthermore, the corresponding results for multilevel Monte Carlo estimators, for which the additional sampling error converges with the same rate as |E[Y - Y_n]|, are presented. Simulations of a stochastic heat equation driven by multiplicative Wiener noise and a geometric Brownian motion are performed which confirm the theoretical results and show the consequences of the presented theory for weak error simulations.
%%%%%%%%%%%%%%%%%%%%%%%%%%%%%%%%%%%%%%%%%%
\end{abstract}

\maketitle
\section{Introduction}%\label{sec:intro}
Weak error analysis for approximations of solutions of stochastic partial differential equations (SPDEs for short) is one of the topics that is currently under investigation within the community of numerical analysis of SPDEs. The goal of weak error analysis is to quantify how well we can approximate a quantity of interest that depends on the solution of an SPDE. While weak convergence rates for equations driven by additive noise are already available (see, e.g., \cite{DP09,AKL16,KLS15,BHS16} and references therein), convergence rates for fully discrete approximations of SPDEs driven by multiplicative noise are still under consideration. First results for semi-discrete approximations in space or time are available (cf., e.g., \cite{Debussche,AL16,CJK14,JK15}) that suggest that one can, as in the case of additive noise, expect a weak convergence rate of twice the order of the strong one, i.e., of mean square convergence. Nevertheless, the simulation of weak error rates has caused problems so far and results are rarely available. First attempts can be found in~\cite{Kruse,Lang4}. 

There are several factors that cause problems in the simulation of weak error rates of SPDE approximations but one of the main reasons is the computational complexity of simulating the solution on a fine grid. To give estimates on quantities of interest, which include the approximation of an expectation, the computational complexity is multiplied by the number of samples in Monte Carlo type methods that are necessary to obtain a reasonable result. We observed in practice that the simulated weak errors in our model problem are very small which in turn requires a far from realistic number of samples in Monte Carlo simulations to get acceptable results.

Motivated by this model problem, we look into the properties of the used estimators. While we are interested in simulating the error $|\E[Y - Y_n]|$ for a sequence of approximations $(Y_n)_{n \in \N_0}$ converging to the real-valued random variable~$Y$, the quantity~$\E[Y_n]$ is analytically not available. The standard approach is to approximate the expectation by a Monte Carlo estimator $E_N[Y_n]$. We show in this manuscript that the additional error when using the estimator $|\E[Y] - E_N[Y_n]|$ instead of the original error is for small errors essentially bounded from above and below by $N^{-1/2}$ in mean square, where $N$ denotes the number of Monte Carlo samples. Furthermore, we consider the estimator $|E_N[Y - Y_n]|$ instead and show that the bounds improve to $N^{-1/2} \Var[Y-Y_n]^{1/2}$, i.e., the number of samples is multiplied by the variance of the error, which can be seen as the strong error in the context of SPDE approximations. Finally, we substitute the Monte Carlo estimators by the corresponding multilevel Monte Carlo estimators and show that the additional error decreases to $|\E[Y - Y_n]|$, i.e., to the error that we are interested in. We confirm the theoretical results in simulations of weak errors of the stochastic heat equation driven by multiplicative noise and a geometric Brownian motion. With the new estimators, we are to the best of our knowledge the first to be able to show weak convergence rates for an SPDE driven by multiplicative noise in simulations.

This manuscript is organized as follows.
In Section~\ref{sec:MLMC-error} we recall Monte Carlo (MC for short) and multilevel Monte Carlo (MLMC) estimators for real-valued random variables. Upper and lower bounds for the approximation of $|\E[Y-Y_n]|$ by different Monte Carlo type estimators are shown.
A short review on SPDEs and their approximation in space, time, and noise is given in Section~\ref{sec:SPDE} and available convergence results are recalled.
Finally, simulation results of strong and weak errors using the estimators introduced in Section~\ref{sec:MLMC-error} for the stochastic heat equation driven by multiplicative noise and a geometric Brownian motion are shown in Section~\ref{sec:simulation}. These confirm the theoretical results of Section~\ref{sec:MLMC-error}.

\section{Monte Carlo versus multilevel Monte Carlo in error analysis}\label{sec:MLMC-error}

In this section we consider upper and lower bounds for the sampling errors that arise when performing weak error simulations in practice. It turns out that it is not surprising that it has not been possible so far to numerically implement weak error analysis for approximations of SPDEs driven by multiplicative noise. Nevertheless, this section is not SPDE specific but formulated more generally for real-valued, square integrable random variables. Quantities of interest in SPDE applications are examples of the framework considered in this section.

Let $(\gO,\cA,P)$ be a probability space and let us for $p \ge 1$ denote by $L^p(\gO;\R)$ the space of all real-valued random variables~$Y$ such that $\|Y\|_{L^p(\gO;\R)}^p = \E[|Y|^p] < + \infty$. We recall that the \emph{Monte Carlo estimator~$E_N$} of a real-valued random variable $Y: \gO \rightarrow \R$ is given by
  \begin{equation}\label{eq:MC-estimator}
   E_N[Y] = N^{-1} \sum_{i=1}^N Y^{(i)},
  \end{equation}
where $(Y^{(i)})_{i=1}^{N}$ is a sequence of independent, identically distributed random variables that have the same law as~$Y$.
Furthermore,  the \emph{multilevel Monte Carlo estimator $E^L$} of a sequence of random variables $(Y_\ell)_{\ell \in \N_0}$ is defined by
  \begin{equation}\label{eq:MLMC-estimator}
   E^L[Y_L] = E_{N_0}[Y_0] + \sum_{\ell=1}^L E_{N_\ell}[Y_\ell - Y_{\ell-1}]
  \end{equation}
for $L \in \N$, where $(N_\ell)_{\ell=0}^{L}$ consists of level specific numbers of samples in the Monte Carlo estimators. For more details on multilevel Monte Carlo methods the reader is referred to the large literature starting with~\cite{H01,G08}.

For later estimates we present the following well-known property of a Monte Carlo estimator, which is a specific form of the law of large numbers and can for example be found in~\cite[Lemma 4.1]{Lang3}.
\begin{lemma}
\label{lem:LLN}
For $N \in \N$ and for $Y \in L^2(\Omega;\R)$ it holds that
\begin{equation*}
\norm{\E\left[Y\right]-E_N[Y]}{L^2(\Omega;\R)} = \frac{1}{\sqrt{N}} \Var[Y]^{1/2} \le \frac{1}{\sqrt{N}} \norm{Y}{L^2(\Omega;\R)}.
\end{equation*}
\end{lemma}

From now on let us consider a square integrable random variable $Y: \gO \rightarrow \R$, i.e., ${Y \in L^2(\gO;\R)}$, and a sequence of approximations $(Y_n)_{n \in \N_0}$ of~$Y$. We assume that it is known that $(Y_n)_{n \in \N_0}$ converges to~$Y$ in the sense that 
  \begin{equation*}
   \lim_{n \rightarrow \infty} |\E[Y-Y_n]| = 0.
  \end{equation*}
In order to estimate convergence rates, one is interested in the simulation of $|\E[Y-Y_n]|$, which usually cannot be done exactly but has to be approximated. If one were 
interested in estimating $\E[Y-Y_n]$, the method of \emph{common random numbers} would tell us that when $Y$ and $Y_n$ are positively correlated, it is better to use an estimator of the form $E_N[Y-Y_n]$ rather than $\E[Y]- E_N[Y_n]$, since the latter has higher variance and both are unbiased. Now, when estimating $|\E[Y-Y_n]|$, the estimators become
   $|\E[Y] - E_N[Y_n]|$
and
   $|E_N[Y-Y_n]|$
instead, neither of which is in general unbiased. In the following lemmas, we therefore show upper and lower bounds on the sampling errors
  \begin{equation*}
    \left\| |\E[Y-Y_n]| - |\E[Y] - E_N[Y_n]|\right\|_{L^2(\gO;\R)}
  \end{equation*}
and
  \begin{equation*}
   \left\| |\E[Y-Y_n]| - |E_N[Y-Y_n]|\right\|_{L^2(\gO;\R)},
  \end{equation*}
 in mean square sense.

\begin{proposition}\label{prop:MC-simple-error-bound}
 The sampling error of approximating $|\E[Y-Y_n]|$ by $|\E[Y] - E_N[Y_n]|$ is bounded from below by
 \begin{equation*}
  \| |\E[Y - Y_n]| - |\E[Y] - E_N[Y_n]| \|_{L^2(\gO;\R)}
    \ge -|\E[Y-Y_n]| + \left(|\E[Y-Y_n]|^2 + N^{-1} \Var[Y_n]\right)^{1/2}
 \end{equation*}
 and from above by
 \begin{equation*}
  \| |\E[Y - Y_n]| - |\E[Y] - E_N[Y_n]| \|_{L^2(\gO;\R)}
    \le N^{-1/2} (\Var[Y_n])^{1/2}.
 \end{equation*}
\end{proposition}

\begin{proof}
 To prove the proposition let us first observe that
  \begin{align*}
   \E[(\E[Y] &- E_N[Y_n])^2]\\
    & = \E[(\E[Y-Y_n] + (\E[Y_n] - E_N[Y_n]))^2]\\
    & = (\E[Y-Y_n])^2 + \E[(\E[Y_n] - E_N[Y_n])^2] + 2 \E[Y-Y_n] \E[\E[Y_n] - E_N[Y_n]],
  \end{align*}
 which implies with Lemma~\ref{lem:LLN} and since $E_N[Y_n]$ is an unbiased estimator of $\E[Y_n]$ that
 \begin{equation*}
  \E[(\E[Y] - E_N[Y_n])^2]
    = |\E[Y-Y_n]|^2 + N^{-1} \Var[Y_n].
 \end{equation*}
 Using this observation we obtain for the squared sampling error that
 \begin{align*}
  \| |\E[Y - Y_n]| &- |\E[Y] - E_N[Y_n]| \|_{L^2(\gO;\R)}^2\\
    & = \E\left[ |\E[Y - Y_n]|^2 + |\E[Y] - E_N[Y_n]|^2 - 2 |\E[Y - Y_n]| \cdot |\E[Y] - E_N[Y_n]| \right]\\
    & = 2 |\E[Y - Y_n]|^2 + N^{-1} \Var[Y_n] - 2 |\E[Y - Y_n]| \E[|\E[Y] - E_N[Y_n]|]\\
    & = N^{-1} \Var[Y_n] - 2 |\E[Y - Y_n]| \E \left[ |\E[Y] - E_N[Y_n]|  - |\E[Y - Y_n]|\right].
 \end{align*}
 To find the lower bound, we observe that
 \begin{align*}
  \E \left[|\E[Y] - E_N[Y_n]| -  |\E[Y - Y_n]|\right]
    & \le \E \left[\left||\E[Y] - E_N[Y_n]| -  |\E[Y - Y_n]|\right|\right] \\
    & \le \E \left[ (|\E[Y - Y_n]| - |\E[Y] - E_N[Y_n]|)^2\right]^{1/2}
 \end{align*}
 by the properties of the expectation and H\"older's inequality.
 Setting
 \begin{equation*}
  E_{n,N} = \| |\E[Y - Y_n]| - |\E[Y] - E_N[Y_n]| \|_{L^2(\gO;\R)},
 \end{equation*}
 which is a positive quantity,
 we therefore obtain the 
 inequality
 \begin{equation*}
  E_{n,N}^2 + 2 |\E[Y-Y_n]| E_{n,N} - N^{-1} \Var[Y_n] \ge 0.
 \end{equation*}
 This is solved using the non-negativity of $E_{n,N}$ by
 \begin{equation*}
  E_{n,N} \ge -|\E[Y-Y_n]| + \left(|\E[Y-Y_n]|^2 + N^{-1} \Var[Y_n]\right)^{1/2},
 \end{equation*}
 which finishes the proof of the lower bound.
 
 For the upper bound we apply the reverse triangle inequality to obtain that
 \begin{equation*}
  \| |\E[Y - Y_n]| - |\E[Y] - E_N[Y_n]| \|_{L^2(\gO;\R)}^2
    \le \|\E[Y_n] - E_N[Y_n]\|_{L^2(\gO;\R)}^2
    = N^{-1} \Var[Y_n],
 \end{equation*}
 where the last step follows from Lemma~\ref{lem:LLN}.
\end{proof}

Having shown that for ${|\E[Y-Y_n]|\ll N^{-1/2}}$ the sampling error is essentially bounded from below and above by $N^{-1/2}$ in terms of the number of Monte Carlo samples when simulating $|\E[Y] - E_N[Y_n]|$, we continue with the sampling error for $|E_N[Y-Y_n]|$. It turns out that this decays for a fixed number of Monte Carlo samples with the square root of the variance of $Y-Y_n$.

\begin{proposition}\label{prop:MC-all-error-bound}
 The sampling error of approximating $|\E[Y-Y_n]|$ by $|E_N[Y-Y_n]|$ is bounded from below by
 \begin{align*}
  \| |\E[Y - Y_n]| &- |E_N[Y-Y_n]| \|_{L^2(\gO;\R)}\\
    & \ge -|\E[Y-Y_n]| + \left(|\E[Y-Y_n]|^2 + N^{-1} \Var[Y-Y_n]\right)^{1/2}
 \end{align*}
 and from above by
 \begin{equation*}
  \| |\E[Y - Y_n]| - |E_N[Y-Y_n]| \|_{L^2(\gO;\R)}
    \le N^{-1/2} (\Var[Y-Y_n])^{1/2}.
 \end{equation*}
\end{proposition}

\begin{proof}
 The proof of the lower bound is performed in the same way as that of Proposition~\ref{prop:MC-simple-error-bound}. The only difference is that we have to simplify $\E[(E_N[Y-Y_n])^2]$ instead of $\E[(\E[Y] - E_N[Y_n])^2]$, which we do in what follows. 
 Therefore, let us observe that due to the properties of the variance and the unbiasedness of the Monte Carlo estimator
 \begin{align*}
  \Var[E_N[Y-Y_n]] 
    & = \E[(E_N[Y-Y_n])^2] - (\E[E_N[Y-Y_n]])^2\\
    & = \E[(E_N[Y-Y_n])^2] - (\E[Y-Y_n])^2.
 \end{align*}
 The independence of the random variables in the Monte Carlo estimator and the Bienaym\'e formula imply that
 \begin{equation*}
  \Var[E_N[Y-Y_n]] 
    = N^{-1} \Var[Y-Y_n],
 \end{equation*}
 which overall leads to
 \begin{equation*}
  \E[(E_N[Y-Y_n])^2]
    = N^{-1} \Var[Y-Y_n] + |\E[Y-Y_n]|^2.
 \end{equation*}
 The proof of the lower bound is finished by applying this formula in the proof of Proposition~\ref{prop:MC-simple-error-bound} and calculating accordingly.
 
 For the upper bound we observe that
  \begin{align*}
   \| |\E[Y - Y_n]| - |E_N[Y-Y_n]| \|_{L^2(\gO;\R)}^2
    & \le \|\E[Y - Y_n] - E_N[Y-Y_n]\|_{L^2(\gO;\R)}^2\\
    & = N^{-1} \Var[Y-Y_n]
  \end{align*}
 again by the reverse triangle inequality and Lemma~\ref{lem:LLN}, which finishes the proof.
\end{proof}

We first remark that the upper bounds in Proposition~\ref{prop:MC-simple-error-bound} and Proposition~\ref{prop:MC-all-error-bound} were already obtained in the context of weak errors for SPDE approximations in~\cite[Proposition~5.4]{P15}.

From the upper bound in Proposition~\ref{prop:MC-simple-error-bound} we learn that the sampling error will not be worse than the Monte Carlo error. At the same time, under the assumption that the quantity of interest ${|\E[Y-Y_n]|\ll N^{-1/2}}$ is very small, we also see that the lower bound implies that we are not able to do better. Therefore, the sampling error is essentially bounded from below and above by the Monte Carlo error. This is not surprising but proves how heavily the simulation relies on the number of Monte Carlo samples for small errors $|\E[Y-Y_n]|$. For relatively cheap computations of samples of $Y_n$ for arbitrarily large $n \in \N$, this is no problem. Nevertheless, in our context, where $Y$ and $Y_n$ are functionals of the solution to an SPDE and its approximation, respectively, the computation is very expensive and the errors $|\E[Y-Y_n]|$ are usually very small compared to their variance. Therefore, it is of no surprise that weak error simulations for SPDEs are still missing in the literature or that many people have failed to simulate them by Monte Carlo methods. One example of a failure with the standard estimator is shown in Section~\ref{sec:simulation}.

Looking into Proposition~\ref{prop:MC-all-error-bound}, we see that we obtain similar upper and lower bounds but instead of $(\Var[Y_n])^{1/2}$, the Monte Carlo error $N^{-1/2}$ is multiplied by $(\Var[Y-Y_n])^{1/2}$, which is usually smaller than $\Var[Y_n]^{1/2}$ and also decrease in~$n$ if $(Y_n)_{n \in \N_0}$ is a sequence of approximations that converges in $L^2(\gO;\R)$ to~$Y$. Therefore we expect faster convergence for the estimator in Proposition~\ref{prop:MC-all-error-bound} than in Proposition~\ref{prop:MC-simple-error-bound}. This is tested and confirmed in Section~\ref{sec:simulation} for a stochastic heat equation driven by multiplicative noise and a geometric Brownian motion.

As a second step in the error analysis, we now consider upper and lower bounds for the sampling error if multilevel Monte Carlo estimators are used instead of the corresponding singlelevel ones, which we discussed above. The results are obtained in a similar way as before.

\begin{proposition}\label{prop:MLMC-simple-error-bound}
 The sampling error of approximating $|\E[Y-Y_L]|$ by $|\E[Y]-E^L[Y_L]|$ is bounded from below by
 \begin{align*}
  \| |\E[Y &- Y_L]| - |\E[Y] - E^L[Y_L]| \|_{L^2(\gO;\R)}\\
    & \ge -|\E[Y-Y_L]| + \left(|\E[Y-Y_L]|^2 + N_0^{-1} \Var[Y_0] + \sum_{\ell=1}^L N_\ell^{-1} \Var[Y_\ell - Y_{\ell-1}]\right)^{1/2}
 \end{align*}
 and from above by
 \begin{equation*}
  \| |\E[Y - Y_L]| - |\E[Y] - E^L[Y_L]| \|_{L^2(\gO;\R)}
    \le \left(N_0^{-1} \Var[Y_0] + \sum_{\ell=1}^L N_\ell^{-1} \Var[Y_\ell - Y_{\ell-1}]\right)^{1/2}.
 \end{equation*}
\end{proposition}

\begin{proof}
 The lower bound is again proven in the same way as in Proposition~\ref{prop:MC-simple-error-bound}, where the only difference is in the computation of $\E[(\E[Y] - E^L[Y_L])^2]$, which we include for completeness. We obtain by the independence of the Monte Carlo estimators on different levels, its unbiasedness, and by Lemma~\ref{lem:LLN} that
 \begin{align*}
  \E[(&\E[Y] - E^L[Y_L])^2]\\
    & = (\E[Y-Y_L])^2 + \E[(\E[Y_0] - E_{N_0}[Y_0])^2] + \sum_{\ell=1}^L \E[(\E[Y_\ell - Y_{\ell-1}] - E_{N_\ell}[Y_\ell - Y_{\ell-1}])^2]\\
    & = |\E[Y-Y_L]|^2 + N_0^{-1} \Var[Y_0] + \sum_{\ell=1}^L N_\ell^{-1} \Var[Y_\ell - Y_{\ell-1}].
 \end{align*}
 For the upper bound we apply after the reverse triangle inequality the same arguments as in the previous computation which yield
 \begin{align*}
  \| |\E[Y - Y_L]| - |\E[Y] - E^L[Y_L]| \|_{L^2(\gO;\R)}^2
    & \le \| \E[Y_L] - E^L[Y_L] \|_{L^2(\gO;\R)}^2\\
    & = N_0^{-1} \Var[Y_0] + \sum_{\ell=1}^L N_\ell^{-1} \Var[Y_\ell - Y_{\ell-1}].
 \end{align*}
 This finishes the proof of the proposition.
\end{proof}

Here the performance of the upper and lower bound depends on the choice of the sample sizes for the different levels of the multilevel Monte Carlo estimator. In Theorem~1 in~\cite{Lang4} it is assumed that there are upper bounds for $|\E[Y-Y_\ell]|$ and $\Var[Y_\ell - Y_{\ell-1}]$. If we assume that we know the errors exactly, we can set $a_\ell = |\E[Y-Y_\ell]|$ and $a_\ell^{2\eta}= \Var[Y_\ell - Y_{\ell-1}]$ in the notation of that theorem in~\cite{Lang4} and choose the sample sizes accordingly. This is made precise in the following corollary that states that the correct choice of samples leads to a sampling error of the same size up to a constant as $|\E[Y-Y_L]|$ which we would like to observe.

\begin{corollary}\label{cor:MLMC-simple-error-bound-with-samples}
 Choosing for a fixed level~$L \in \N$ the sample sizes in the multilevel Monte Carlo estimator
 $N_0 = \ceil{|\E[Y-Y_L]|^{-2}}$ and
 $N_\ell = \ceil{|\E[Y-Y_L]|^{-2} \Var[Y_\ell - Y_{\ell-1}] \ell^{1+\epsilon}}$ for $\ell=1,\ldots,L$ and some $\epsilon > 0$
 it holds that
 \begin{equation*}
  \| |\E[Y - Y_L]| - |\E[Y] - E^L[Y_L]| \|_{L^2(\gO;\R)}
    \ge ( 2^{-1/2}(3 + \Var[Y_0])^{1/2} - 1) |\E[Y-Y_L]|
 \end{equation*}
 and
 \begin{equation*}
  \| |\E[Y - Y_L]| - |\E[Y] - E^L[Y_L]| \|_{L^2(\gO;\R)}
    \le (\Var[Y_0] + \gz(1+\epsilon))^{1/2} |\E[Y-Y_L]|,
 \end{equation*}
 where $\gz$ denotes the Riemann zeta function.
 Therefore,
 \begin{equation*}
  \| |\E[Y - Y_L]| - |\E[Y] - E^L[Y_L]| \|_{L^2(\gO;\R)}
    \simeq  |\E[Y-Y_L]|,
 \end{equation*}
 i.e., the sampling error converges with the same rate as $|\E[Y-Y_L]|$.
\end{corollary}

\begin{proof}
 Let us first observe that for $x \ge 1$ we have that $x \le \ceil{x} \le 2 x$. This implies with the given choices of $N_\ell$, $\ell=0,\ldots,L$, that
 \begin{equation*}
  N_0^{-1} \Var[Y_0] + \sum_{\ell=1}^L N_\ell^{-1} \Var[Y_\ell - Y_{\ell-1}]
    \le |\E[Y-Y_L]|^2 \left(\Var[Y_0] + \sum_{\ell=1}^L \ell^{-(1+\epsilon)}\right)
 \end{equation*}
 as well as
 \begin{equation*}
  N_0^{-1} \Var[Y_0] + \sum_{\ell=1}^L N_\ell^{-1} \Var[Y_\ell - Y_{\ell-1}]
    \ge 2^{-1} |\E[Y-Y_L]|^2 \left(\Var[Y_0] + \sum_{\ell=1}^L \ell^{-(1+\epsilon)}\right).
 \end{equation*}
 We observe next that
 \begin{equation*}
  1 \le \sum_{\ell=1}^L \ell^{-(1+\epsilon)} \le \gz(1+\epsilon)
 \end{equation*}
 for all $L \in \N$ and plug the obtained inequalities into the equations in Proposition~\ref{prop:MLMC-simple-error-bound} to finish the proof of the corollary.
\end{proof}
For completeness we include the equivalent statement to Proposition~\ref{prop:MC-all-error-bound} for the multilevel Monte Carlo estimator, but we remark that it is of no practical interest. This is due to the fact that in particular $E_{N_0}[Y-Y_0]$ has to be computed, i.e., many samples of the exact solution must be generated, which is computationally too expensive and destroys the idea of multilevel Monte Carlo methods.
\begin{proposition}\label{prop:MLMC-all-error-bound}
 The sampling error of approximating $|\E[Y-Y_L]|$ by $|E^L[Y-Y_L]|$ is bounded from below by
 \begin{align*}
  \| |\E[Y &- Y_L]| - |E^L[Y-Y_L]| \|_{L^2(\gO;\R)}\\
    & \ge -|\E[Y-Y_L]| + \left(|\E[Y-Y_L]|^2 +  N_0^{-1} \Var[Y-Y_0] + \sum_{\ell=1}^L N_\ell^{-1} \Var[Y_\ell - Y_{\ell-1}]\right)^{1/2}
 \end{align*}
 and from above by
 \begin{equation*}
  \| |\E[Y - Y_L]| - |E^L[Y-Y_L]| \|_{L^2(\gO;\R)}
    \le \left( N_0^{-1} \Var[Y-Y_0] + \sum_{\ell=1}^L N_\ell^{-1} \Var[Y_\ell - Y_{\ell-1}]\right)^{1/2}.
 \end{equation*}
\end{proposition}

\begin{proof}
 The proof is again performed in the same way as that of Proposition~\ref{prop:MC-simple-error-bound}, where it is essential to derive
 \begin{align*}
  \E[(E^L[Y-Y_L])^2]
    = |\E[Y-Y_L]|^2 + N_0^{-1} \Var[Y-Y_0] + \sum_{\ell=1}^L N_\ell^{-1} \Var[Y_\ell - Y_{\ell-1}].
 \end{align*}
 Due to the repetition in techniques and the rather theoretical nature of the claim we leave further details of the proof to the interested reader.
\end{proof}

\section{Approximation of mild SPDE solutions}
\label{sec:SPDE}

In this section we employ the framework of \cite{Kruse} in a simplified setting and recall some of the results of that monograph. We provide a noise approximation result for a stochastic evolution equation with multiplicative noise in the very end of this section.
Let $H=L^2([0,1];\R)$ be the space of square integrable functions on the unit interval $[0,1]$ with inner product~$\inpro{\cdot}{\cdot}{H}$ given by $\inpro{v}{w}{H}=\int^1_0 v(x)w(x) \, \dd x$, which is a real separable Hilbert space with orthonormal basis~$(e_j)_{j \in \N}$,
where $e_j(x) = \sqrt{2} \sin(j \pi x) $. 
Let $Q \in L(H)$, where $L(H)$ is the space of all bounded linear operators from $H$ to $H$, be a self-adjoint, positive definite operator of trace class. We denote by $H_0 = Q^{1/2}(H)$ the Hilbert space with inner product $\inpro{\cdot}{\cdot}{H_0}~=~\langle Q^{-1/2}\cdot, Q^{-1/2}\cdot\rangle_{H}$,
where $Q^{-1}$ denotes the pseudo-inverse of~$Q$. Furthermore, we let $(\Omega, \cA, (\cF)_{t \in [0,T]}, P)$ be the extension of the probability space in Section~\ref{sec:MLMC-error} with a normal filtration. We assume that $W=(W(t))_{t \ge 0}$ is an $(\cF)_{t \in [0,T]}$~-~adapted $Q$-Wiener process.
In this framework we consider for $t \in [0,T]$ the stochastic partial differential equation 
\begin{equation}
\label{eq:stochasticheat}
\dd X(t) - \Delta X(t) \, \dd t= G(X(t)) \, \dd W(t)
\end{equation}
with initial condition $X(0) = X_0$,
which we refer to as the one-dimensional heat equation driven by multiplicative Wiener noise. 
Here we denote by $\Delta$ the Laplace operator with zero boundary conditions. It holds that $-\Delta$ has eigenbasis $(e_j)_{j \in \N}$ with eigenvalues $\lambda_j = j^2 \pi^2$
and $\Delta$ generates a $C_0$-semigroup of contractions denoted by $S=(S(t))_{t\ge 0}$ on $H$. The fractional operator $(-\Delta)^{r/2} : \dot{H}^{r} \to H$ 
has domain $\dot{H}^{r}=\{ v \in H : \norm{v}{r}^2 = \sum^\infty_{j=1} \lambda_j^r \inpro{v}{e_j}{H}^2 \}$ for $r\ge0$. It holds that $\dot{H}^{r}$ is a separable Hilbert space when equipped with the inner product 
\begin{equation*}
\inpro{\cdot}{\cdot}{r} = \inpro{(-\Delta)^{\frac{r}{2}} \cdot}{(-\Delta)^{\frac{r}{2}} \cdot}{H} . 
\end{equation*}

We impose further assumptions on the parameters of~\eqref{eq:stochasticheat} in what follows, which are stronger than Assumptions 2.13 and 2.17 in~\cite[Chapter 2]{Kruse} and hence guarantee the existence and uniqueness of a mild solution 
\begin{equation}
\label{eq:mildsolution}
X(t) = S(t) X_0 + \int^t_0 S(t-s)G(X(s)) \, \dd W(s).
\end{equation}
\begin{assumption} Assume that the parameters of~\eqref{eq:stochasticheat} satisfy:
\label{assumptions:1}
\hfill
\begin{enumerate}[label=(\roman*)]
\item 
\label{assumptions:1:Q}
The trace class operator $Q$ is defined through the relation $Q e_j = \mu_j e_j$ where $\mu_j = C_\mu j^{-\eta}$ for two constants $C_\mu > 0$ and $\eta > 1$.
\item \label{assumptions:1:G} 
Fix a parameter $r \in [0,1)$. The mapping $G : H \to L_{\text{HS}}(H_0;H) $ satisfies for a constant $C>0$
\begin{enumerate}[label=(\alph*)]
\item  $G(v) \in L_{\textsc{HS}}(H_0;\dot{H}^r)$ for all $v \in \dot{H}^r$,
\item  $\norm{(-\Delta)^{r/2} G(v)}{L_\textsc{HS}(H_0;H)} \le C (1+ \norm{v}{r})$ for all $v \in \dot{H}^r$,
\item  $\norm{G(v_1)-G(v_2)}{L_\textsc{HS}(H_0;H)} \le C \norm{v_1-v_2}{H}$ for all $v_1,v_2 \in H$, and 
\item %\label{assumptions:1:G4} 
$\norm{G(v) e_j}{H} \le C (1+\norm{v}{H})$ for all basis vectors $e_j \in H$ and $v \in H$,
\end{enumerate}
where $L_{\text{HS}}(H_0;H)$ denotes the space of Hilbert--Schmidt operators from $H_0$ to~$H$.
\item  Assume that $X_0 \in \dot{H}^{1+r}$ is a deterministic initial value.
\end{enumerate}
\end{assumption}

We remark that in the notation of Assumption~\ref{assumptions:1}\ref{assumptions:1:Q},
we may write $W$ in terms of its \KL expansion
\begin{equation}
\label{eq:KLexpansion}
W(t) = \sum_{j=1}^{\infty} \mu_j^{\frac{1}{2}} \beta_j (t) e_j,
\end{equation}
where $(\beta_j)_{j \in \N}$ is a sequence of independent, real-valued Wiener processes.

In order to be able to simulate realizations of the mild solution~\eqref{eq:mildsolution}, we approximate it by a Galerkin finite element method in space and an implicit Euler--Maruyama scheme in time. For this, let $(V_h)_{h \in (0,1]}$ be the nested sequence of finite-dimensional subspaces, where $V_{h}\subset\dot{H}^1\subset H$ is given for each $h$ by
the family of continuous functions that are piecewise linear on the intervals $[x_j,x_{j+1}]$ of an equidistant partition $(x_j)_{j=0}^{N_h}$ of $[0,1]$ defined by $ x_j = j h$ for $j= 0,1 \ldots ,N_h$ and zero on the boundary, where $N_h=1/h$ is assumed to be an integer.

We define the discrete operator $-\Delta_h : V_h \to V_h$ on each $v_h \in V_h$ by letting $-\Delta_h v_h$ be the unique element of $V_h$ such that
$$
\inpro{-\Delta v_h}{w_h}{H} = \inpro{v_h}{w_h}{1} = \inpro{-\Delta_h v_h}{w_h}{H}
$$ 
for all $w_h \in V_h$. 

For the discretization in time we define a uniform time grid $(t_j)_{j=1}^{N_k}$ with time step size $k \in (0,1]$ by $t_j = jk$ for $j=0,1, \ldots, N_k$, where $N_k = T/k$ is again assumed to be an integer. To be able to implement this approximation scheme on a computer, one must also consider a noise approximation, i.e., an approximation of the $Q$-Wiener process~$W$. One way of doing this is to truncate the \KL expansion~\eqref{eq:KLexpansion}, which has earlier been considered for example in~\cite{KLL10,KLNS11,BL12,BL12_AMO,JR15}. This leads to the Wiener process
\begin{equation*}
W^\kappa(t) = \sum_{j=1}^{\kappa} \mu_j^{\frac{1}{2}} \beta_j (t) e_j.
\end{equation*}
The fully discrete implicit Euler-Maruyama approximation $X^j_{\kappa,h} $ of $X(t_j)$ is then given in recursive form by
\begin{equation}\label{eq:truncatedfullydiscretesolution}
%\label{eq:fullydiscretepde}
X^j_{\kappa,h} - X^{j-1}_{\kappa,h} - k(\Delta_h X^j_{\kappa,h}) = P_h G(X^{j-1}_{\kappa,h}) (W^\kappa(t_j) - W^\kappa(t_{j-1}))% \text{ for } j=1,\ldots,N_k \\
\end{equation}
for $j=1,\ldots,N_k$ with initial condition $X^0_h = P_h X_0$, where $P_h$ denotes the orthogonal projection onto~$V_h$.
This scheme converges strongly with $\Op(k^{1/2})$ to the mild solution, which is stated in the following theorem that combines \cite[Theorem 3.14]{Kruse} with an additional noise approximation.
\begin{theorem}%[{\cite[Theorem~4.3]{P15}}]
\label{thm:noiseapproximation}
Under Assumption~\ref{assumptions:1} for fixed $r \in [0,1)$ and the couplings $k^{1/2} \simeq h^{1+r}$, $\kappa \simeq h^{- \beta}$ with ${\beta(\eta -1)} = {2(1+r)}$, it holds that for all $p \ge 2$ there exists a constant $C>0$ such that for all $k \in (0,1]$ and $j=1, \ldots, N_k$ 
\begin{equation*}
\| X(t_j) - X_{\kappa,h}^j \|_{L^p(\Omega;H)} \le C k^{1/2}.
\end{equation*}
\end{theorem}
\begin{proof} 
The proof is the same as that of~\cite[Theorem~4.3]{P15} except that
\begin{equation*}
  \norm{G_{\kappa,h} (s) e_j}{H}^2 \le C( \norm{X^j_{\kappa,h}}{H}^2 + 1)
\end{equation*}
is used in the estimate of~$\textsc{II}_b$ in that proof, which does not change the final bound on~~$\textsc{II}_b$.
\end{proof}

While strong approximations are well understood and proven for the considered framework, weak convergence rates in the sense of bounds of the error
\begin{equation*}
 |\E[\phi(X_h^j)] - \E[\phi(X(t_j))]|,
\end{equation*}
where $\phi$ is a smooth functional, are still missing. Nevertheless, results on the convergence of approximations of SPDEs driven by additive noise (cf., e.g., \cite{AKL16,KLS15}) as well as semi-discrete approximations of either space (cf., e.g., \cite{Debussche,AL16,CJK14}) or time (cf., e.g., \cite{JK15}) suggest that the weak convergence rate is twice the strong one. This is in accordance with the results obtained in our simulations in Section~\ref{sec:simulation}, where we consider the case $\phi=\norm{\cdot}{H}^2$. Our choice of the test function~$\phi$ ensures that the weak error is bounded by the strong error (cf., \cite[Chapter 6]{Kruse} and \cite[Chapter 5]{P15}).

\section{Simulation}\label{sec:simulation}

In this section simulation results that combine the theory of Section~\ref{sec:MLMC-error} and Section~\ref{sec:SPDE} are presented, i.e., weak errors of an SPDE approximation are computed with Monte Carlo and multilevel Monte Carlo methods and compared. 
In the setting of Assumption~\ref{assumptions:1}, we now fix the parameters $T=1$, $X(0,x)=x-x^2$ and $C_\mu=\eta=5$ and consider two choices of the operator~$G: H \to \hsspacetwo{H_0}{H}$.

The first operator $G_1$ is defined for $v \in H$ and $v_0 \in H_0$ by
\begin{align*}
G_1(v) v_0 = \sum^\infty_{j=1} \inpro{v}{e_j}{H} \inpro{v_0}{e_j}{H} e_j,
\end{align*}
while the so called \emph{Nemytskii type} operator $G_2$ is defined for $x \in [0,1]$ by
\begin{equation*}
(G_2(v) v_0)[x] = \sin (v(x))v_0(x).
\end{equation*}
These operators satisfy Assumption~\ref{assumptions:1}\ref{assumptions:1:G} for $r \in [0, \infty)$ and $r \in [0,1/2)$ respectively, shown in~\cite[Section 6.4]{Kruse} and~\cite[Example 2.23]{Kruse}. 

The choice $G=G_1$ admits an analytical solution~$X$ of~\eqref{eq:stochasticheat} and for this the identities 
\begin{align}
\label{eq:G1exactsolution}
X(t) &= \sum^\infty_{j = 1}  \inpro{X_0}{e_j}{H} \exp\left(- ( \lambda_j + \frac{\mu_j}{2})t + \mu_j^\frac{1}{2} \beta_j (t) \right) e_j, 
\\
\label{eq:G1exactnorm}
\norm{X(t)}{H}^2 &= \sum^\infty_{j = 1} \inpro{X_0}{e_j}{H}^2 \exp\left(- ( 2 \lambda_j + \mu_j)t + 2 \mu_j^\frac{1}{2} \beta_j (t) \right),
\\
\label{eq:G1exactexpectednorm}
\E[{\norm{X(t)}{H}^2}] &= \sum^\infty_{j = 1} \inpro{X_0}{e_j}{H}^2 \exp\left( (- 2  \lambda_j + {\mu_j})t\right)
\end{align}
hold for all $t \in [0,T]$. The numerical approximations $X^{N_k}_{\kappa,h}$ of $X(T)$ are now computed by first setting $X^0_{\kappa,h} = I_h X_0$ and then recursively solving the numerical equation 
\begin{equation*}
X^j_{\kappa,h} - X^{j-1}_{\kappa,h} - k(\Delta_h X^j_{\kappa,h}) = G^h_i(X^{j-1}_{\kappa,h}) \left(W^\kappa(t_j)-W^\kappa(t_{j-1})\right) \text{ for } j=1,\ldots,N_k,
\end{equation*}
where the interpolation operator $I_h: H \to V_h$ is defined by 
\begin{equation*}
I_h f(x) = \sum^{N_h-1}_{j=1} f(x_j) \Phi_j (x) 
\end{equation*}
and for $G^h_i : V_h \to \hsspacetwo{H_0}{V_h}$, $i \in \{1,2\}$, we set
\begin{equation*}
G^h_1 (v_h) v_0 = \sum^\kappa_{j=1} \inpro{v_h}{I_h e_j}{H} \inpro{v_0}{e_j}{H} I_h e_j 
\end{equation*}
as well as
\begin{equation*}
G^h_2 (v_h) v_0 = I_h G_2(v_h) v_0.  
\end{equation*}
The replacement of the operator $P_h$ with $I_h$ mirrors to a large extent the setting of \cite[Chapter 6]{Kruse}, where the author notes that this is quite common in practice, and we note that the simulation results below indicate that the order of convergence is not affected. For refinement levels $\ell \in \N$ we set $k_\ell = 2^{-2 \ell}$ and, as a shorthand notation, we write $\hat{X}_\ell =X^{N_{k_\ell}}_{\kappa_\ell, h_\ell}$ for the end time evaluation of~\eqref{eq:truncatedfullydiscretesolution} with $h_\ell =k_\ell^{1/2}$ and $\kappa_\ell =k^{-1/2}_\ell$. 
With these choices, by Theorem~\ref{thm:noiseapproximation}, we expect a strong convergence rate of order~$1/2$ in time and by the usual rule of thumb that the weak rate of convergence is twice the strong one, we expect a weak rate of order~$1$.

The following simulations were performed on the Glenn cluster at Chalmers Centre for Computational Science and Engineering (C3SE) using the MATLAB Distributed Computing Server\texttrademark. In all of them, we approximate $\norm{v}{H}^2$ for $v \in H$ by $N_{h_\ell}^{-1} \sum^{N_{h_\ell}}_{j=1} v(x_j)^2$.

In Figure~\ref{P15fig2} an approximation of the strong error $\norm{X(T)-\hat{X}_\ell}{L^2(\gO;H)}$, i.e.,
\begin{equation*}
 E_N\left[\bignorm{X (T)-\hat{X}_\ell (T)}{H}^2\right]^{\frac{1}{2}},
\end{equation*}
is calculated for levels $\ell = 1,\ldots, 7$, where we replace the exact solution $X$ with a \emph{reference solution} $\tilde{X}_L$. 
For $G=G_1$ the reference solution $\tilde{X}_L$ is given by~\eqref{eq:G1exactsolution} truncated at $j = \kappa_L=2^L$, while for $G=G_2$ we let $\tilde{X}_L = \hat{X}_L$, since we do not have access to an analytic solution. We let $L=8$ and take $N=12 \cdot 10^3$ samples. 
It should be noted that the same realizations of the $N$ $Q$-Wiener processes are used for the error computations on all levels. The observed error rate is asymptotically $\Op(k^{1/2})$ and therefore consistent with Theorem~\ref{thm:noiseapproximation}.

\begin{figure}[ht]
   \centering
     \subfigure[Strong error for $\ell=1,\ldots,7$ with $N=12 \cdot 10^3$ samples.\label{P15fig2}]{\includegraphics[width=0.49\textwidth]{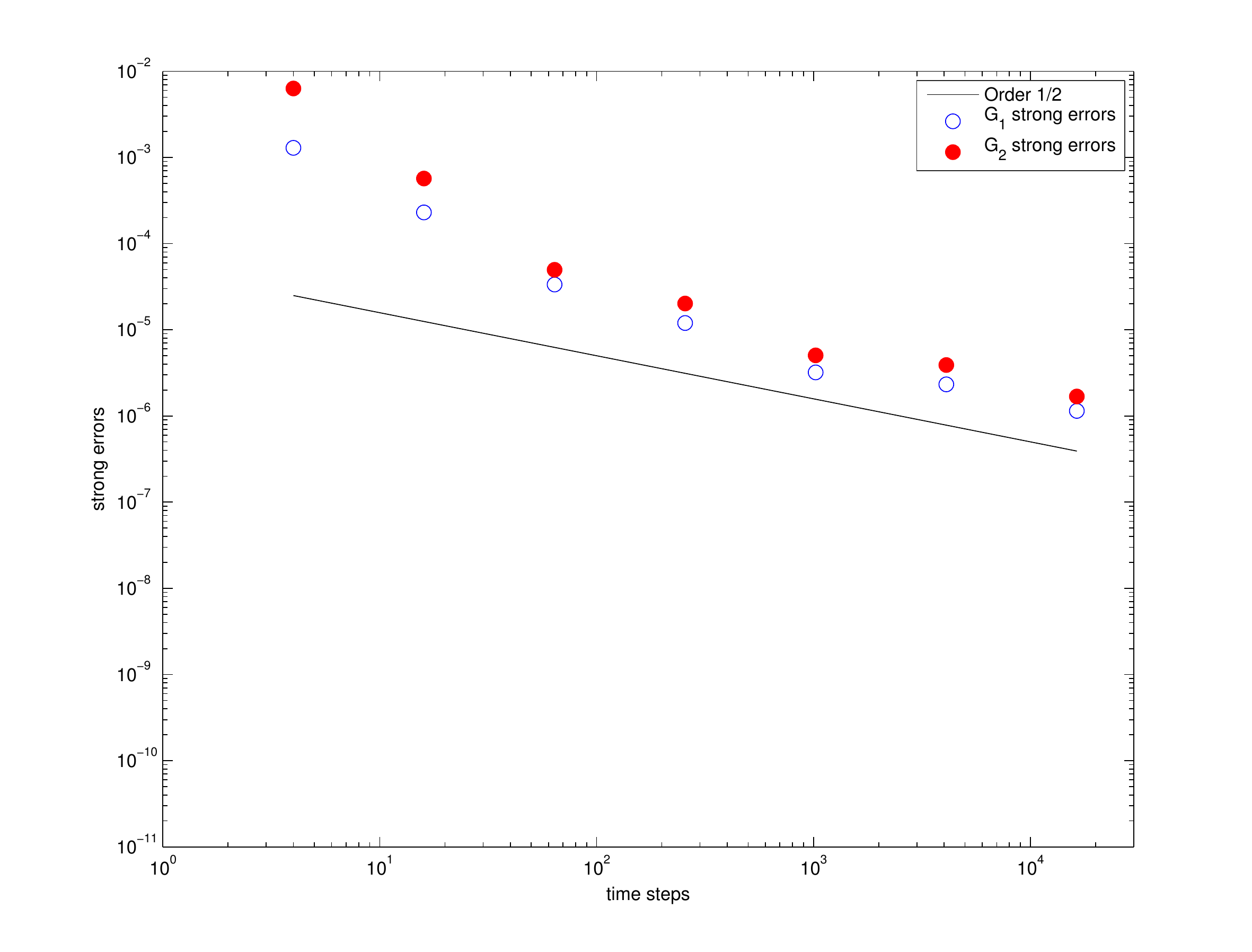}} 
 \subfigure[Weak error for $\ell=1,\ldots,5$ with $N=3 \cdot 10^3$ samples. \label{P15fig4}]{\includegraphics[width=0.49\textwidth]{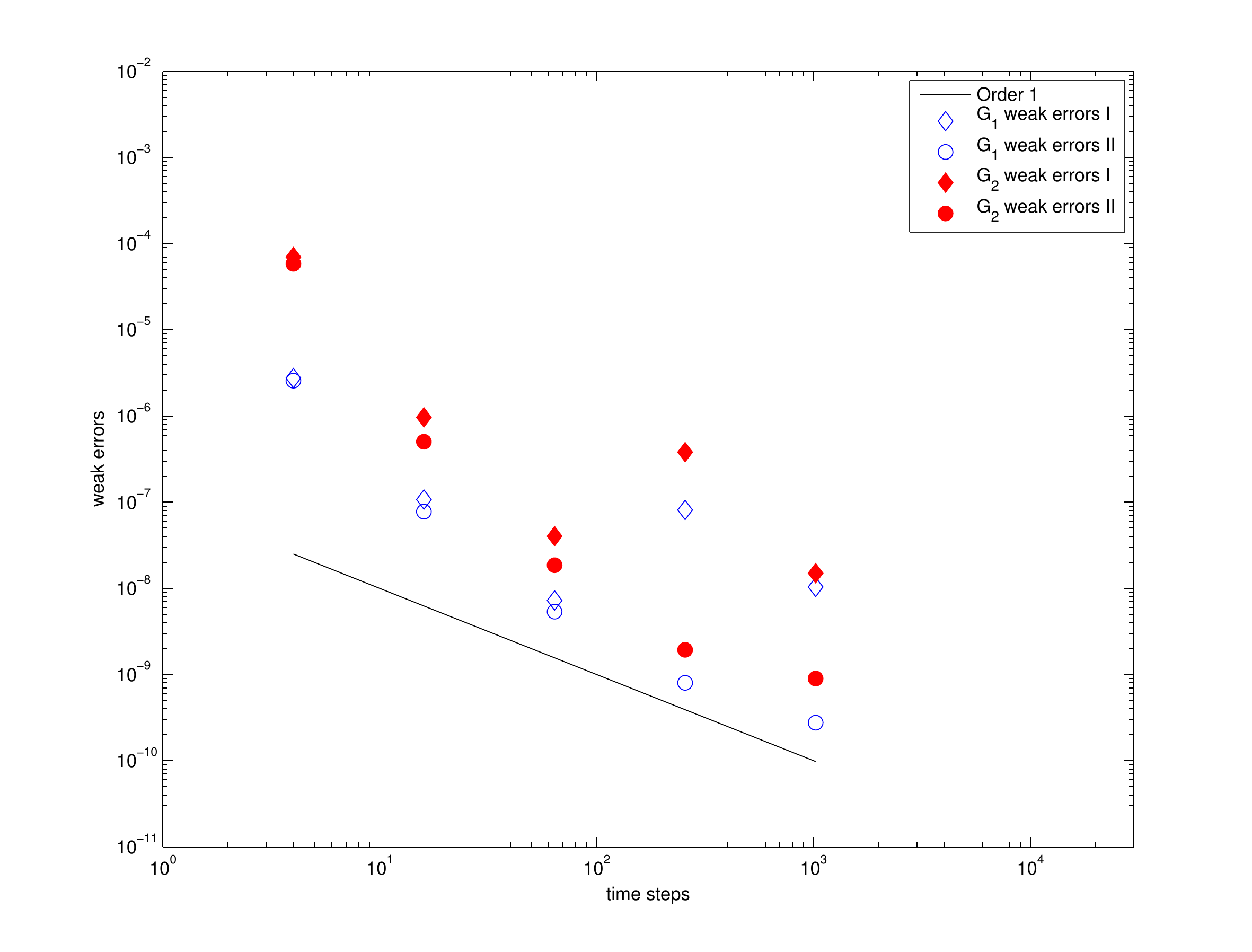}}
 \caption{Approximations of strong and weak errors using singlelevel Monte Carlo estimators.}
\end{figure}

Next, we estimate the weak error and compare the performance of the singlelevel Monte Carlo estimators of Propositions~\ref{prop:MC-simple-error-bound} and \ref{prop:MC-all-error-bound}. In Figure~\ref{P15fig4} the weak error is approximated with $N=3 \cdot 10^3$ samples and $\ell = 1,\ldots, 5$. For the weak error approximation according to Proposition~\ref{prop:MC-simple-error-bound} with $Y=\norm{X(T)}{H}^2$ and $Y_\ell = \norm{\hat{X}_\ell}{H}^2$ we set
\begin{equation*}
 e_{1,\ell} = \left| \bigE{\norm{X(T)}{H}^2}-E_N\left[ \norm{\hat{X}_\ell}{H}^2\right] \right|,
\end{equation*}
which we refer to as \emph{error of type~I} in what follows. Here we calculate $E_N\left[ \norm{\hat{X}_\ell}{H}^2\right]$,  $\ell = 1,\ldots, 5$, using separate sets of realizations of $N$ $Q$-Wiener processes for each level $\ell$. Furthermore, we replace $\bigE{\norm{X(T)}{H}^2}$ by \eqref{eq:G1exactexpectednorm} evaluated at $t=T$ and truncated at $j=10^6$ in the case of $G=G_1$. In the case of $G=G_2$, we replace it instead by a reference solution $E_N \left[ \norm{\hat{X}_L}{H}^2\right]$ with $L=8$ and $N=10^4$, which is calculated on an independent set of $Q$-Wiener processes. For the weak error approximation according to Proposition~\ref{prop:MC-all-error-bound}
\begin{equation*}
 e_{2,\ell} = \left|E_N \left[ \norm{X(T)}{H}^2-\norm{\hat{X}_\ell}{H}^2 \right] \right|,
\end{equation*}
 called \emph{error of type~{II}} in what follows, the samples of $\|X(T)\|_H^2-\|\hat{X}_\ell\|_H^2$, $\ell = 1,\ldots, 5$, are computed on the same set of $N$ $Q$-Wiener processes. In the case of $G=G_1$, we replace the exact solution $\|X(T)\|_H^2$ with \eqref{eq:G1exactnorm} evaluated at $t=T$ and truncated at $j=h_L^{-1}$. For $G=G_2$ we use again a reference solution $\norm{\hat{X}_L}{H}^2$. 
 In said figure, i.e., Figure~\ref{P15fig4}, we show the average of these estimators
 \begin{equation*}
   e^M_{i,\ell} = M^{-1}\sum^M_{j=1} e^j_{i, \ell}
 \end{equation*}
for $i=1,2$, where $(e^j_{i, \ell})^M_{j=1}$ is a sequence of independent copies of $e_{i, \ell}$ and $M=10$, to see how they perform in general.
While the errors of type~{II} supersede and then approach twice the strong order of convergence, the errors of type~I do not. This is due to the limitation of the convergence by the number of Monte Carlo samples from below as shown in Proposition~\ref{prop:MC-simple-error-bound}, that is to say, with a constant sample size $N$ we get a sampling error proportional to $\Var(\|\hat{X}_\ell\|_H^2)$. This indicates that the observation of weak convergence results with a naive Monte Carlo estimator cannot be computed satisfactory in an acceptable time even for such a relatively easy example, where details on the computational times are collected for all estimators at the end of the example. 

For the errors of type~{II} the rate of convergence seems to decrease for the last level. This is explained by Proposition~\ref{prop:MC-all-error-bound}---in contrast to the the type~{I} errors the sampling error is proportional to ${\Var(\|X(T)\|_H^2-\| \hat{X}_\ell\|_H^2)^{1/2}}$ which is bounded by the strong error (measured in~$L^4(\gO;\R)$) and therefore the rate of convergence starts to resemble $\Op(k^{1/2})$.

For the last set of simulations in Figure~\ref{P15fig6} we compare singlelevel type~I errors to type~I weak errors obtained using the multilevel Monte Carlo estimator~\eqref{eq:MLMC-estimator} instead of the naive Monte Carlo approximation~\eqref{eq:MC-estimator}. For $\ell \in \N_0$ we set $k_\ell = k_0 2^{-2 \ell}, h_\ell = k_\ell^{1/2}, \kappa_\ell=k^{-1/2}_\ell$ to obtain a series of fully discrete approximations $\check{X}_\ell=X^{N_{k_\ell}}_{\kappa_\ell, h_\ell}$, where for computational reasons we let $k_0 = 2^{-2}$. In the errors 
\begin{equation*}
  \hat{e}_L =\left| \bigE{\norm{X(T)}{H}^2}-E^L\left[ \norm{\check{X}_L}{H}^2\right] \right|,
\end{equation*}
we replace $\smallE{\norm{X(T)}{H}^2}$ by the same quantities as for the type I errors in Figure~\ref{P15fig4}. We note that each multilevel estimate $E^L\left[ \norm{\check{X}_L}{H}^2\right]$ is generated independently of one another so that the type~{I} errors become a natural comparison. 
Figure~\ref{P15fig6} shows the multilevel error approximations for $L = 1,\ldots, 5$ and the corresponding singlelevel errors of type I with sample sizes $N=3 \cdot 10^3$. We again show an average of these and let $M=100$.
From the figure we observe that the multilevel Monte Carlo estimators show the weak convergence rate, while the errors of type~I fail. The latter errors are dominated by the sampling error as before, but for the multilevel estimator we know from Corollary~\ref{cor:MLMC-simple-error-bound-with-samples} that the sampling error is bounded from above and below by the weak error. This explains why this approach succeeds in showing the expected weak convergence rates. The fact that the errors of type~I are in total smaller than those obtained by the multilevel Monte Carlo simulation is due to the larger constant in the error estimates of Corollary~\ref{cor:MLMC-simple-error-bound-with-samples} which can be reduced by enlarging the overall number of samples in the multilevel Monte Carlo method.

\begin{figure}[ht]
   \centering
     \includegraphics[width=0.75\textwidth]{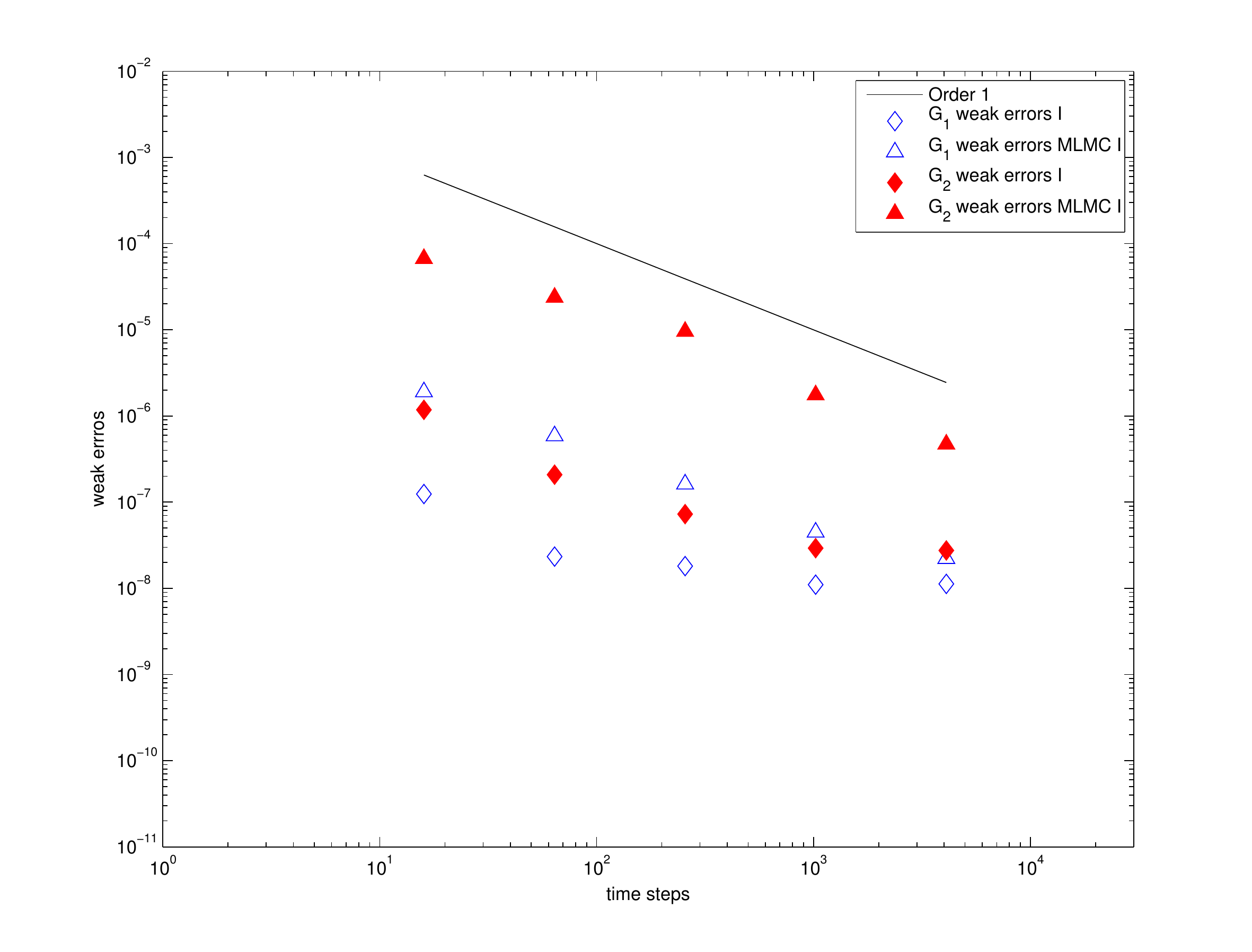}
\caption{Average of $100$ realizations of MLMC computed at levels $L=1,\ldots,5$ with the corresponding MC estimates of the weak error.} 
 \label{P15fig6}
\end{figure}

To give the reader an idea of the computational complexity of the shown convergence plots, we include the computing times, rounded off to the nearest hour, for 8~computing nodes with a total of 128~cores on the Glenn cluster of C3SE. The strong error plot Figure~\ref{P15fig2} cost 40~hours, while the weak error simulation of type~{II} in Figure~\ref{P15fig4} took 102~hours. The reference solution used for the weak singlelevel errors of type~I and the multilevel Monte Carlo errors for $G=G_2$ were computed in 13~hours. The costs for the weak errors of type I in Figure~\ref{P15fig4} are negligible (i.e., less than one hour). The computation of the multilevel errors in Figure~\ref{P15fig6} took 32~hours, while it took just 2~hours for the singlelevel errors. It is important to note that the computation of the type~I errors was quite cheap since we could reuse the reference solution for $G=G_2$. One should also be aware that we would have needed to increase the number of samples by at least a factor of $2^{12}$ to see the weak convergence for the type~{I} error of $G=G_1$, which would have increased the computational time to more than 8000~hours.

In conclusion we have seen in this section that the simulation of weak errors of SPDE approximations causes severe problems which we already expected out of the theory in Section~\ref{sec:MLMC-error}. The use of a multilevel Monte Carlo estimator and a modified Monte Carlo estimator finally led to the expected weak convergence plots due to a faster convergence of the sampling error caused by the approximation of the expectation shown theoretically in Section~\ref{sec:MLMC-error}. It is important to point out at this point here that these are to our knowledge the first successful simulations of weak errors for SPDEs driven by multiplicative noise. 

Due to the limitations in computational complexity, we further illustrate the theoretical results of Section~\ref{sec:MLMC-error} with the simulation of an ordinary stochastic differential equation. The relative cheapness of such a simulation allows us to make the consequences of Section~\ref{sec:MLMC-error} even clearer than above. Let us therefore consider in what follows the easy example of a geometric Brownian motion in one dimension, i.e., the stochastic differential equation
\begin{equation}\label{eq:gBM}
 \dd X(t) = \mu X(t) \, \dd t + \sigma X(t) \, \dd W(t)
\end{equation}
with initial condition $X(0) = X_0 \in \R$ and $t \in [0,T]$, where $\mu, \sigma \in \R$ and $W$ denotes a one-dimensional Brownian motion. The solution to the geometric Brownian motion is known to be
\begin{equation*}
 X(t) = X_0 \exp((\mu - \gs^2/2)t + \gs W(t) )
\end{equation*}
and the second moment can be computed explicitly to be
\begin{equation*}
 \E[|X(t)|^2] = X_0^2 \exp( (2\mu + \gs^2)t).
\end{equation*}
For the approximation let us consider an equidistant time discretization $(t_j)_{j=1}^{N_k}$ with time step size $k \in (0,1]$ by $t_j = jk$ for $j=0,1, \ldots, N_k$, where $N_k = T/k$ is assumed to be an integer. The Euler--Maruyama scheme is then given by the recursion
\begin{equation*}
 X^{j} = (1 + k \mu + \gs(W(t_j) - W(t_{j-1}))) X^{j-1},
\end{equation*}
and $X^0 = X_0$, where $X^j$ denotes the approximation of $X(t_j)$.
It is known that this scheme converges for the geometric Brownian motion with strong order $\Op(k^{1/2})$, i.e.,
\begin{equation*}
 \E[|X(t_j) - X^j|^2]^{1/2} \le C k^{1/2},
\end{equation*}
and with weak order $\Op(k)$, i.e., for sufficiently smooth test functions $\phi:\R \rightarrow \R$ it holds that
\begin{equation*}
 |\E[\phi(X(t_j)) - \phi(X^j)]| \le C k,
\end{equation*}
where $j=0,1, \ldots, N_k$. Here the constant~$C$ does not depend on~$k$.

In the simulation of the four estimators from Section~\ref{sec:MLMC-error}, let us consider the geometric Brownian motion~\eqref{eq:gBM} with $\mu = -0.5$, $\gs = 1$, $X_0 = 1$, and $T = 0.5$. Furthermore, let $N = 10^{4.5}$ be the number of samples in the Monte Carlo estimator and $M = 20$ in the notation of the previous example with the same estimators as before . Then we obtain on time grids with $2^j + 1$ grid points for $j=1,\ldots,8$ the convergence plots for the simulated strong and weak errors which are presented in Figure~\ref{fig:gBM}.
\begin{figure}[ht]
   \centering
     \subfigure[Strong error for $\ell=1,\ldots,8$.\label{fig:gBM_strong}]{\includegraphics[width=0.49\textwidth]{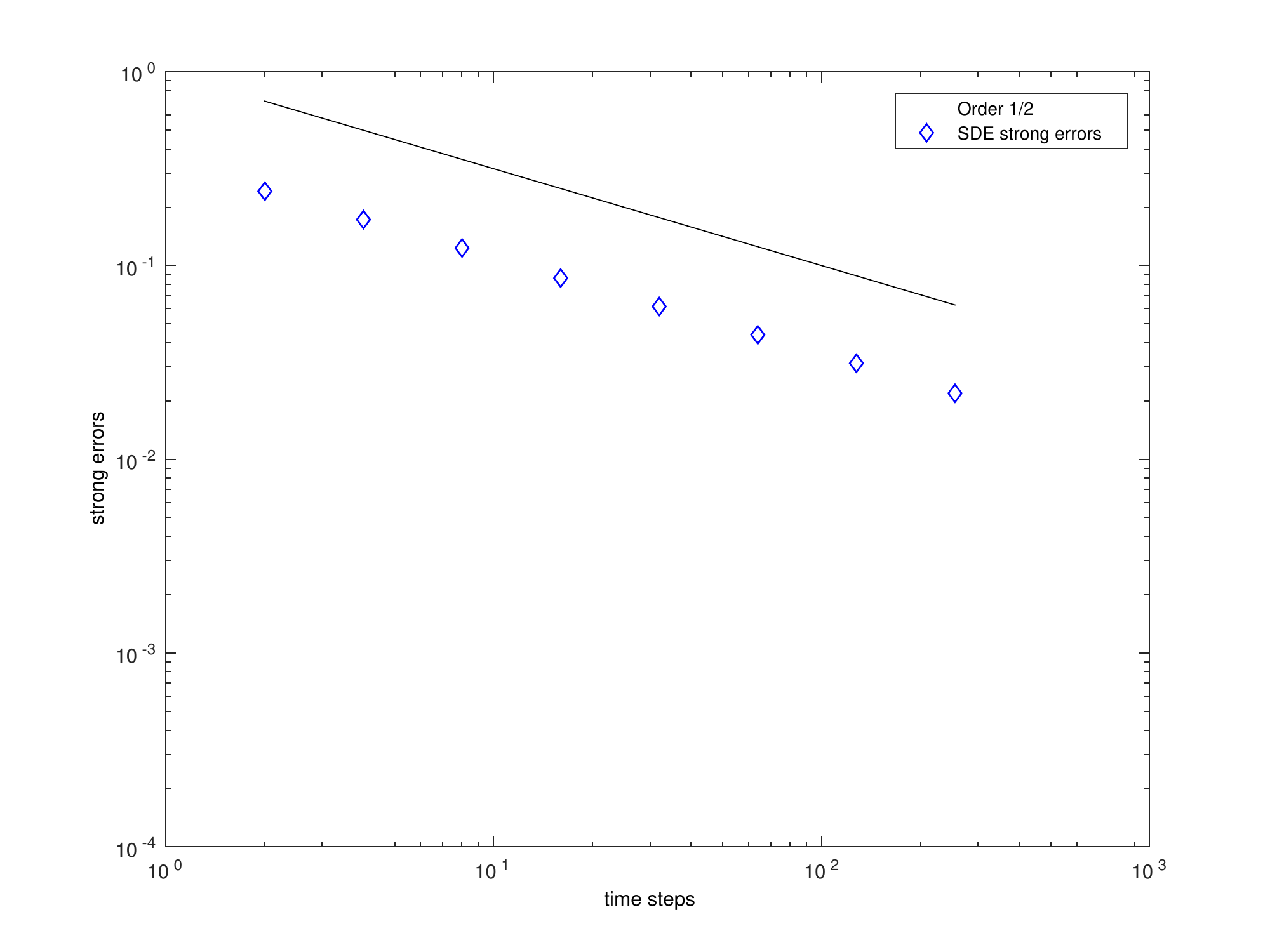}} 
 \subfigure[Weak error for $\ell=1,\ldots,8$. \label{fig:gBM_weak}]{\includegraphics[width=0.49\textwidth]{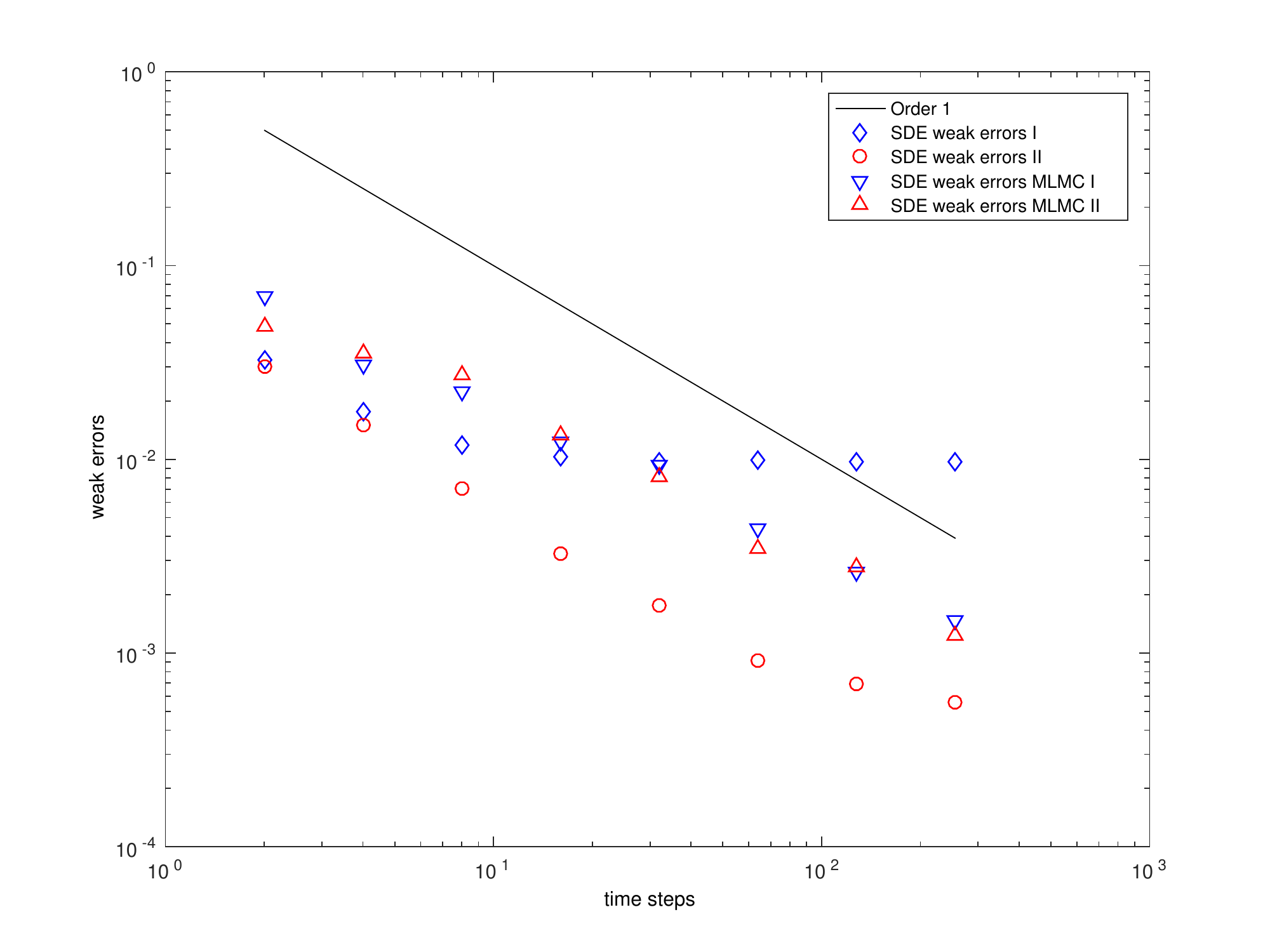}} 
 \caption{Approximations of strong and weak errors of the Euler--Maruyama scheme for the geometric Brownian motion.\label{fig:gBM}}
\end{figure}
In the weak error simulation, the function $\phi(x) = |x|^2$ is used. We observe that the strong error in Figure~\ref{fig:gBM_strong} converges as expected with $\Op(k^{1/2})$. In Figure~\ref{fig:gBM_weak} one sees that the type~I estimator of Proposition~\ref{prop:MC-simple-error-bound} $|\E[|X(T)|^2] - E_N[|X^{N_k}|^2]|$, which just does a Monte Carlo simulation on the approximate solution, only converges on the first two grid points with the desired order before as in the theory the Monte Carlo error dominates. At the same time, the type~{II} estimator $|E_N[|X(T)|^2 - |X^{N_k}|^2]|$, which was considered in Proposition~\ref{prop:MC-all-error-bound}, behaves a lot better. It converges with the desired order up to the last two points, where the strong order of convergence dominates as predicted by the theory. Both the multilevel Monte Carlo estimator of type~{I} from Proposition~\ref{prop:MLMC-simple-error-bound} and of type~{II} from Proposition~\ref{prop:MLMC-all-error-bound} converge with the desired order of convergence but the absolute errors are larger due to the larger constant in the overall error. This easy example, where all correct values were known and could be used for the computations of the exact solutions, shows clearly the behaviour that we expected from the theoretical upper and lower bounds on weak error estimators in Section~\ref{sec:MLMC-error}.

\bibliographystyle{hplain}
\bibliography{MCM-Proceedings-update}

\end{document}